\documentstyle[11pt,twoside,amsthm]{article}
\theoremstyle{definition}
\newtheorem{definition}{Definition}[section]
\newtheorem{notation}[definition]{Notation}
\newtheorem{example}[definition]{Example}

\theoremstyle{plain}
\newtheorem{theorem}[definition]{Theorem}
\newtheorem{lemma}[definition]{Lemma}
\newtheorem{proposition}[definition]{Proposition}
\newtheorem{corollary}[definition]{Corollary}

\newcommand{\beq}{\begin{equation}}
\newcommand{\eeq}{\end{equation}}
\newcommand{\bdfn}{\begin{definition}}
\newcommand{\edfn}{\end{definition}}
\newcommand{\bthm}{\begin{theorem}}
\newcommand{\ethm}{\end{theorem}}
\newcommand{\bprop}{\begin{proposition}}
\newcommand{\eprop}{\end{proposition}}
\newcommand{\bcor}{\begin{corollary}}
\newcommand{\ecor}{\end{corollary}}
\newcommand{\blem}{\begin{lemma}}
\newcommand{\elem}{\end{lemma}}
\newcommand{\bex}{\begin{example}}
\newcommand{\eex}{\end{example}}
\newcommand{\bxc}{\begin{exercise}}
\newcommand{\exc}{\end{exercise}}
\newcommand{\bntn}{\begin{notation}}
\newcommand{\entn}{\end{notation}}

\newcommand{\be}{\begin{enumerate}}
\newcommand{\ee}{\end{enumerate}}
\newcommand{\bce}{\begin{center}}
\newcommand{\ece}{\end{center}}
\newcommand{\bi}{\begin{itemize}}
\newcommand{\ei}{\end{itemize}}

\newcommand{\bt}{\begin{tabular}}
\newcommand{\et}{\end{tabular}}

\newcommand{\Ra}{\Rightarrow}

\newcommand{\si}{\wedge}
\newcommand{\sau}{\vee}
\newcommand{\ba}{\begin{array}} 
\newcommand{\ea}{\end{array}}

\begin{document}

\title{Dense Elements and Classes of Residuated Lattices}
\author{Claudia MURE\c{S}AN}
\maketitle

\begin{center}
{\em Dedicated to my beloved grandparents, Floar\u{a}-Marioara Mure\c{s}an, Elena Mircea and Ion Mircea}
\end{center}

\begin{abstract}
In this paper we study the dense elements and the radical of a residuated lattice, residuated lattices with lifting Boolean center, simple, local, semilocal and quasi-local residuated lattices. BL-algebras have lifting Boolean center; moreover, Glivenko residuated lattices which fulfill the equation $(\neg \, a\rightarrow \neg \, b)\rightarrow \neg \, b=(\neg \, b\rightarrow \neg \, a)\rightarrow \neg \, a$ have lifting Boolean center.

\vspace*{12pt} 
\noindent AMS {\bf Subject Classification:} Primary 06F35, Secondary 03G10.

\vspace*{12pt} 
\end{abstract}

\section{Introduction}

\hspace*{11pt} In this paper we study properties of residuated lattices related to the dense elements, the radical, the lifting Boolean center, as well as several classes of residuated lattices.

In Section \ref{prelim1} of the article we recall some definitions and facts about residuated lattices that we use in the sequel: rules of calculus, de\-fi\-ni\-tions of several classes of residuated lattices, definitions of the regular elements, the Boolean center, the dense elements and the radical of a re\-si\-du\-a\-ted lattice, definitions of filters and spectral topologies and first properties of the defined notions.

In Section \ref{denselem} we study the dense elements and the radical of a residuated lattice. We make this study for arbitrary residuated lattices, as well as for certain classes of residuated lattices.

In Section \ref{liftprop} we study the property of lifting Boolean center for residuated lattices and indicate certain classes of residuated lattices which have lifting Boolean center.

In Section \ref{locsemiloc} we study local, semilocal and simple residuated lattices and we define and study quasi-local residuated lattices. We also show the relations between these classes of residuated lattices.

\section{Preliminaries}
\label{prelim1}

\hspace*{11pt} A {\em (commutative) residuated lattice} is an algebraic structure $(A,\vee ,\wedge ,\odot ,$\linebreak $\rightarrow ,0,1)$, with the first 4 operations binary and the last two constant, such that $(A,\vee ,\wedge ,0,1)$ is a bounded lattice, $(A,\odot ,1)$ is a commutative monoid and the following property, called {\em residuation}, is satisfied: for all $a,b,c\in A$, $a\leq b\rightarrow c\Leftrightarrow a\odot b\leq c$, where $\leq $ is the partial order of the lattice $(A,\vee ,\wedge ,0,1)$.

In the following, unless mentioned otherwise, let $A$ be a residuated lattice. 

We define two additional operations on $A$: for all $a\in A$, we denote $\neg \,a=a\rightarrow 0$, and, for all $a,b\in A$, we denote $a\leftrightarrow b=(a\rightarrow b)\wedge (b\rightarrow a)$.

Let $a\in A$ and $n\in {\rm I\! N}^{*}$. We shall denote by $a^{n}$ the following element of $A$: $\underbrace{a\odot \ldots \odot a}_{n\ {\rm of}\ a}$. We also denote $a^{0}=1$.

Let $a\in A$. The {\em order} of $a$, in symbols $ord(a)$, is the smallest $n\in {\rm I\! N}^{*}$ such that $a^{n}=0$. If no such $n$ exists, then $ord(a)=\infty $.

\begin{lemma}{\rm \cite{pic}} Let $a,b,c,d\in A$. Then:

\noindent (i) $a\leq b$ iff $a\rightarrow b=1$; 

\noindent (ii) $a\leq \neg \, \neg \, a$;

\noindent (iii) $\neg \, \neg \, \neg \, a=\neg \, a$;

\noindent (iv) $\neg \, 0=1$, $\neg \, 1=0$;

\noindent (v) if $a\leq b$ and $c\leq d$ then $a\odot c\leq b\odot d$.
\label{`leqres`}
\end{lemma}

An element $a$ of $A$ is said to be {\em regular} iff $\neg \, \neg \, a=a$. The set of all regular elements of $A$ is denoted ${\rm Reg}(A)$. $A$ is said to be {\em involutive} iff it verifies: for all $a\in A$, $\neg \, \neg \, a=a$, that is: ${\rm Reg}(A)=A$.

An {\em MTL-algebra} is a residuated lattice $A$ that satisfies {\em the preliniarity equation}: for all $a,b\in A$, $(a\rightarrow b)\vee (b\rightarrow a)=1$. An involutive MTL-algebra is called an {\em IMTL-algebra}. A {\em BL-algebra} is an MTL-algebra $A$ that satisfies {\em the divisibility equation}: for all $a,b\in A$, $a\wedge b=a\odot (a\rightarrow b)$. 

An {\em MV-algebra} is an algebra $(A,\oplus,\neg \, ,0)$ with one binary 
operation $\oplus$, one unary operation $\neg \, $ and one constant 0 such that: $(A,\oplus,0)$ is a commutative monoid and, for all $a,b\in A$, $\neg \, \neg \, a=a$, $a\oplus \neg \, 0=\neg \, 0$, $\neg \, (\neg \, a\oplus b)\oplus b=\neg \, (\neg \, b\oplus a)\oplus a$. If $A$ is an MV-algebra, then the binary operations $\odot$, $\si$, $\sau$, $\rightarrow $ and the constant 1 are defined by the following relations: for all $a,b\in A$, $a\odot b=\neg \, (\neg \, a\oplus \neg \, b)$, $a\si b=(a\oplus \neg \, b)\odot b$ , $a\sau b=(a\odot \neg \, b)\oplus b$, $a\rightarrow b=\neg\, a\oplus b$, $1=\neg \, 0$. According to \cite[Theorem 3.2, page 99]{pic}, MV-algebras are exactly the involutive BL-algebras. For a detailed exposition of MV-algebras see \cite{cdom}. 

Again, in what follows, unless mentioned otherwise, let $A$ be a residuated lattice.

We denote by $B(A)$ the {\em Boolean center of $A$}, that is the set of all complemented elements of the lattice $(A,\vee ,\wedge ,0,1)$. By \cite[Lemma 1.12]{pic}, the complements of the elements in the Boolean center of a residuated lattice are unique. For any element $e$ from the Boolean center of a residuated lattice, we denote by $e^{\prime }$ its complement. By \cite[Lemma 1.13]{pic}, any element $e$ from the Boolean center of a residuated lattice satisfies: $e^{\prime }=\neg \, e$.

\bprop \label{car-B(A)}{\rm \cite[Proposition 1.17, page 15]{pic}} For any $e\in B(A)$, we have: $e\odot e=e$ and $e=\neg \, \neg \, e$.
\eprop

\begin{proposition}{\rm \cite[Proposition 1.16, page 14]{pic}} $e\in B(A)$ iff $e\vee \neg \, e=1$.
\label{B(A)=}
\end{proposition}

By \cite[Corollary 1.15]{pic}, the Boolean center of $A$, with the operations induced by those of $A$, is a Boolean algebra.

Let $A_{1}$, $A_{2}$ be residuated lattices and $f:A_{1}\rightarrow A_{2}$ a morphism of residuated lattices. We denote by $B(f):B(A_{1})\rightarrow B(A_{2})$ the restriction of $f$ to $B(A_{1})$. It is known and immediate that $B(f)$ is well defined and it is a morphism of Boolean algebras, hence $B$ is a covariant functor between the category of residuated lattices and the category of Boolean algebras.

A nonempty subset $F$ of $A$ is called a {\em filter of $A$} iff it satisfies the following conditions:

\noindent (i) for all $a,b\in F$, $a\odot b\in F$;

\noindent (ii) for all $a\in F$ and all $b\in A$, if $a\leq b$ then $b\in F$. 

The set of all filters of $A$ is denoted ${\cal{F}}(A)$. A filter $F$ of $A$ is said to be {\em proper} iff $F\neq A$.

A subset $F$ of $A$ is called a {\em deductive system of $A$} iff it satisfies the following conditions:

\noindent (i) $1\in F$;

\noindent (ii) for all $a,b\in A$, if $a,a\rightarrow b\in F$ then $b\in F$.

By \cite[Remark 1.11, page 19]{pic}, a nonempty subset $F$ of $A$ is a filter iff it is a deductive system. 

A proper filter $P$ of $A$ is called a {\em prime filter} iff, for all $a,b\in A$, if $a\vee b\in P$, then $a\in P$ or $b\in P$. The set of all prime filters of $A$ is called {\em the (prime) spectrum of $A$} and is denoted ${\rm Spec}(A)$.

For any $X\subseteq A$, we shall denote $S(X)=\{P\in {\rm Spec}(A)|X\subseteq \!\!\!\!\!\!/\ P\}$. For any $a\in A$, $S(\{a\})$ will be denoted $S(a)$. The family $\{S(X)|X\subseteq A\}$ is a to\-po\-lo\-gy on ${\rm Spec}(A)$, having the basis $\{S(a)|a\in A\}$ (\cite[Proposition 2.1]{eu2}). The family $\{S(X)|X\subseteq A\}$ is called {\em the Stone topology of $A$}.

A filter $M$ of $A$ is called a {\em maximal filter} iff it is a maximal element of the set of all proper filters of $A$. The set of all maximal filters of $A$ is called {\em the maximal spectrum of $A$} and is denoted ${\rm Max}(A)$.

For any $X\subseteq A$, we shall denote $S_{\rm Max}(X)=\{M\in {\rm Max}(A)|X\subseteq \!\!\!\!\!\!/\ M\}$. For any $a\in A$, $S_{\rm Max}(\{a\})$ will be denoted $S_{\rm Max}(a)$. It is a known fact that any maximal filter of a residuated lattice is a prime filter. It follows that the family $\{S_{\rm Max}(X)|X\subseteq A\}$ is a to\-po\-lo\-gy on ${\rm Max}(A)$, having the basis $\{S_{\rm Max}(a)|a\in A\}$. This is the topology induced on ${\rm Max}(A)$ by the Stone topology.





Let $F$ be a filter of $A$. For all $a,b\in A$, we denote $a\equiv b({\rm mod}\ F)$ and say that {\em $a$ and $b$ are congruent modulo $F$} iff $a\leftrightarrow b\in F$. $\equiv ({\rm mod}\ F)$ is a congruence relation on $A$. The quotient residuated lattice with respect to the congruence relation $\equiv ({\rm mod}\ F)$ is denoted $A/F$ and its elements are denoted $a/F$, $a\in A$. A consequence of Lemma \ref{`leqres`}, (i), is that, for all $a,b\in A$, if $a\leq b$, then $a/F\leq b/F$.

$0,1\in {\rm Reg}(A)$. By \cite{cito1}, if $a,b\in A$, then: $\neg \, \neg \, (\neg \, \neg \, a\rightarrow \neg \, \neg \, b)=\neg \, \neg \, a\rightarrow \neg \, \neg \, b$. It follows that ${\rm Reg}(A)$ is closed with respect to $\rightarrow $ (because, by the above, for all $a,b\in {\rm Reg}(A)$, $a\rightarrow b=\neg \, \neg \, (a\rightarrow b)$, hence $a\rightarrow b\in {\rm Reg}(A)$). 

We define on ${\rm Reg}(A)$ the following o\-pe\-ra\-tions: for all $a,b\in {\rm Reg}(A)$, $a\odot ^{*}b=\neg \, \neg \, (a\odot b)$, $a\vee ^{*}b=\neg \, \neg \, (a\vee b)$ and $a\wedge ^{*}b=\neg \, \neg \, (a\wedge b)$. We will always refer to the operations above when studying the algebraic structure of ${\rm Reg}(A)$.

$A$ is said to be {\em Glivenko} iff it satisfies: for all $a\in A$, $\neg \, \neg \, (\neg \, \neg \, a\rightarrow a)=1$. In the case of Glivenko MTL-algebras, the operations $\vee ^{*}$ and $\wedge ^{*}$ coincide with $\vee $ and $\wedge $, respectively (\cite{cito2}).

\begin{proposition}{\rm \cite[Theorem 2.1, page 163]{cito2}} The following are equivalent:

\noindent (i) $A$ is Glivenko;

\noindent (ii) $({\rm Reg}(A),\vee ^{*},\wedge ^{*},\odot ^{*},\rightarrow ,0,1)$ is an involutive residuated lattice and $\neg \, \neg \, :A\rightarrow {\rm Reg}(A)\ :\ a\rightarrow \neg \, \neg \, a$ is a surjective morphism of residuated lattices.
\label{`glivnegneg`}
\end{proposition}


Note that ${\rm Reg}(A)=\neg \, A$ (trivial).

\begin{proposition}{\rm \cite[Lemma 2.4, (2), page 116]{cito1}} Let $A$ be a Glivenko re\-si\-du\-a\-ted lattice. Then ${\rm Reg}(A)$ is an MV-algebra (with the operation $\oplus $ defined by: for all $a,b\in {\rm Reg}(A)$, $\neg \, a\oplus \neg \, b=\neg \, (a\odot ^{*}b)$) iff $A$ satisfies the following equation: for all $a,b\in A$, $(\neg \, a\rightarrow \neg \, b)\rightarrow \neg \, b=(\neg \, b\rightarrow \neg \, a)\rightarrow \neg \, a$.
\label{regmv}
\end{proposition}

\bprop\label{B(A)=B(MV(A))}
Let $A$ be a Glivenko MTL-algebra. Then $B(A)=$\linebreak $B({\rm Reg}(A))$.
\eprop 
\begin{proof}
Obviously, $B({\rm Reg}(A))\subseteq B(A)$. For the converse set inclusion, let us consider an element $e\in B(A)$. By Propositions \ref{car-B(A)} and \ref{`glivnegneg`}, $e=\neg \, \neg \, e\in B(A)\cap {\rm Reg}(A)$ and $\neg \, e\in {\rm Reg}(A)=\neg \, A$. These two properties imply that $e\in B({\rm Reg}(A))$.\end{proof}

An element $a$ of $A$ is said to be {\em dense} iff $\neg \,a=0$. We denote by $Ds(A)$ the set of the dense elements of $A$. The intersection of all ma\-xi\-mal filters of $A$ is called {\em the radical of $A$} and is denoted ${\rm Rad}(A)$. $Ds(A)$ is a filter of $A$ and $Ds(A)\subseteq {\rm Rad}(A)$ (\cite{fre}, \cite{kow}, \cite[Theorem 1.61]{pic}).

\begin{proposition}{\rm \cite[Theorem 3.4, page 117]{cito1}} Let $A$ be a Glivenko residuated lattice and $\theta :A/Ds(A)\rightarrow {\rm Reg}(A)$, for all $a\in A$, $\theta (a/Ds(A))=\neg \, \neg \, a$. Then $\theta $ is an isomorphism of residuated lattices and the following diagram is commutative:

\begin{center}
\begin{picture}(70,70)(0,0)
\put(2,43){$A$}
\put(12,48){\vector(1,0){38}}
\put(27,51){$p_{A}$}
\put(53,43){$A/Ds(A)$}
\put(56,40){\vector(0,-1){26}}
\put(59,25){$\theta $}
\put(10,40){\vector(4,-3){40}}
\put(10,20){$\neg \, \neg \, $}
\put(53,2){${\rm Reg}(A)$}
\end{picture}
\end{center}

\noindent where, for all $a\in A$, $p_{A}(a)=a/Ds(A)$.
\label{`dreg`}
\end{proposition}

%



\begin{proposition}{\rm \cite[Lemma 3.3, page 117]{cito1}} Let $A$ be a residuated lattice. Then the following are equivalent:

\noindent (i) $A/Ds(A)$ is involutive; 

\noindent (ii) $A$ is Glivenko.
\label{`glivinv`}
\end{proposition}

\section{Dense Elements and the Radical of a Residuated Lattice}
\label{denselem}

\hspace*{11pt} In this section we study the set of the dense elements and the radical of an arbitrary residuated lattice, as well as those of certain classes of residuated lattices.

\begin{proposition}
Let $A$ be a residuated lattice, $M\in {\rm Max}(A)$ and $a\in A$. Then:

\noindent (i) $a/Ds(A)\in M/Ds(A)$ iff $a\in M$;

\noindent (ii) ${\rm Max}(A/Ds(A))=\{N/Ds(A)|N\in {\rm Max}(A)\}$;

\noindent (iii) $a/Ds(A)\in {\rm Rad}(A)/Ds(A)$ iff $a\in {\rm Rad}(A)$;

\noindent (iv) ${\rm Rad}(A/Ds(A))={\rm Rad}(A)/Ds(A)$;

\noindent (v) ${\rm Max}(A)$ and ${\rm Max}(A/Ds(A))$ are homeomorphic topological spaces.
\label{`maxrad`}
\end{proposition}

\begin{proof}
(i) If $a/Ds(A)\in M/Ds(A)$ then there exists $b\in M$ such that\linebreak $a/Ds(A)=b/Ds(A)$, so $a\leftrightarrow b\in Ds(A)\subseteq {\rm Rad}(A)\subseteq M$, and, since $b\in M$, this implies $a\in M$ (remember that the notion of filter is equivalent to that of deductive system). The converse implication is obvious.

\noindent (ii) Obviously, ${\cal{F}}(A/Ds(A))=\{F/Ds(A)|F\in {\cal{F}}(A),F\supseteq Ds(A)\}$. Let $N\in {\cal{F}}(A)$ such that $N\supseteq Ds(A)$. By (i), we have: $N\in {\rm Max}(A)$ iff $N\neq A$ and $(\forall F\in {\cal{F}}(A)\setminus \{A\})\, N\subseteq F\Rightarrow N=F$ iff $N/Ds(A)\neq A/Ds(A)$ and $(\forall F/Ds(A)\in {\cal{F}}(A/Ds(A))\setminus \{A/Ds(A)\})\, N/Ds(A)\subseteq F/Ds(A)\Rightarrow N/Ds(A)=F/Ds(A)$ iff $N/Ds(A)\in {\rm Max}(A/Ds(A))$, hence the desired equality.

\noindent (iii) By (i), $a/Ds(A)\in {\rm Rad}(A)/Ds(A)$ iff $(\forall N\in {\rm Max}(A))\, a/Ds(A)\in N/Ds(A)$ iff $(\forall N\in {\rm Max}(A))\, a\in N$ iff $a\in {\rm Rad}(A)$.

\noindent (iv) According to (ii), $\displaystyle {\rm Rad}(A/Ds(A))=\bigcap _{N\in {\rm Max}(A)}N/Ds(A)=$\linebreak $\displaystyle \left( \bigcap _{N\in {\rm Max}(A)}N\right) /Ds(A)={\rm Rad}(A)/Ds(A)$.

\noindent (v) Let us define $h:{\rm Max}(A)\rightarrow {\rm Max}(A/Ds(A))$, for all $N\in {\rm Max}(A)$, $h(N)=N/Ds(A)$. By (ii), this function is well defined and surjective.

Let $M_{1},M_{2}\in {\rm Max}(A)$ such that $h(M_{1})=h(M_{2})$, that is: $M_{1}/Ds(A)=M_{2}/Ds(A)$. Let $b\in A$. (i) and this last equality show that: $b\in M_{1}$ iff $b/Ds(A)\in M_{1}/Ds(A)$ iff $b/Ds(A)\in M_{2}/Ds(A)$ iff $b\in M_{2}$. Hence $M_{1}=M_{2}$, so $h$ is injective.

It remains to prove that $h$ is continuous and open. Let $b\in A$. By using in turn (ii) and (i), we get: $S_{\rm Max}(b/Ds(A))=\{N/Ds(A)|N\in {\rm Max}(A),$\linebreak $b/Ds(A)\notin N/Ds(A)\}=\{N/Ds(A)|N\in {\rm Max}(A),b\notin N\}=\{N/Ds(A)|N\in S_{\rm Max}(b)\}=\{h(N)|N\in S_{\rm Max}(b)\}=h(S_{\rm Max}(b))$. So $h$ is open. This and the injectivity of $h$ prove that $h^{-1}(S_{\rm Max}(b/Ds(A)))=S_{\rm Max}(b)$. So $h$ is continuous. 

Therefore $h$ is a homeomorphism of topological spaces.\end{proof}

\begin{lemma}
\noindent (i) For any residuated lattice $A$, ${\rm Rad}(A)=\{a\in A|(\forall \, n\in {\rm I\!N}^{*})\, (\exists \, m\in {\rm I\!N}^{*})\, \neg \, ((\neg \, (a^{n}))^{m})=1\}$.

\noindent (ii) Let $A$ be a residuated lattice and $B$ a subalgebra of $A$. Then: ${\rm Rad}(B)=B\cap {\rm Rad}(A)$.

\noindent (iii) Let $\{A_{i}|i\in I\}$ be a family of residuated lattices. Then: $\displaystyle {\rm Rad}(\prod _{i\in I}A_{i})=\prod _{i\in I}{\rm Rad}(A_{i})$.
%
\label{`rad`}
\end{lemma}
\begin{proof}
(i) can be found in \cite{fre}, \cite{kow}, \cite[Theorem 1.61, page 35]{pic}. (ii) and (iii) are obvious consequences of (i).\end{proof}

In \cite{eu3}, it is proven that, For any residuated lattice $A$ and any filter $F$ of $A$ such that $F\subseteq Ds(A)$, we have $Ds(A/F)=(Ds(A))/F$. The proposition below shows that this is not the case for any filter.

\begin{proposition}
There exist residuated lattices $A$ and filters $F$ of $A$ such that $Ds(A/F)\neq (Ds(A))/F$.
\label{`d`}
\end{proposition}

\begin{proof}
Let us consider the following example of residuated lattice from \cite[Section 15.2.2]{ior2}: $A=\{0,a,b,c,d,1\}$, with the residuated lattice structure presented below:

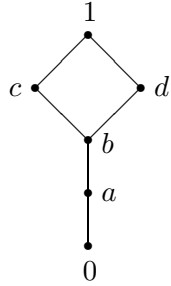
\begin{center}
\begin{picture}(80,115)(0,0)
\put(40,15){\circle*{3}}
\put(40,35){\circle*{3}}
\put(40,55){\circle*{3}}
\put(20,75){\circle*{3}}
\put(60,75){\circle*{3}}
\put(40,95){\circle*{3}}
\put(40,15){\line(0,1){40}}
\put(40,55){\line(-1,1){20}}
\put(40,55){\line(1,1){20}}
\put(40,95){\line(-1,-1){20}}
\put(40,95){\line(1,-1){20}}
\put(38,2){$0$}
\put(38,100){$1$}
\put(10,72){$c$}
\put(65,72){$d$}
\put(45,32){$a$}
\put(45,50){$b$}
\end{picture}
\end{center}

\begin{center}
\begin{tabular}{cc}
\begin{tabular}{c|cccccc}
$\rightarrow $ & $0$ & $a$ & $b$ & $c$ & $d$ & $1$ \\ \hline
$0$ & $1$ & $1$ & $1$ & $1$ & $1$ & $1$ \\
$a$ & $d$ & $1$ & $1$ & $1$ & $1$ & $1$ \\
$b$ & $a$ & $a$ & $1$ & $1$ & $1$ & $1$ \\
$c$ & $0$ & $a$ & $d$ & $1$ & $d$ & $1$ \\
$d$ & $a$ & $a$ & $c$ & $c$ & $1$ & $1$ \\

$1$ & $0$ & $a$ & $b$ & $c$ & $d$ & $1$
\end{tabular}
& \hspace*{11pt}
\begin{tabular}{c|cccccc}
$\odot $ & $0$ & $a$ & $b$ & $c$ & $d$ & $1$ \\ \hline
$0$ & $0$ & $0$ & $0$ & $0$ & $0$ & $0$ \\
$a$ & $0$ & $0$ & $0$ & $a$ & $0$ & $a$ \\
$b$ & $0$ & $0$ & $b$ & $b$ & $b$ & $b$ \\
$c$ & $0$ & $a$ & $b$ & $c$ & $b$ & $c$ \\
$d$ & $0$ & $0$ & $b$ & $b$ & $d$ & $d$ \\
$1$ & $0$ & $a$ & $b$ & $c$ & $d$ & $1$
\end{tabular}
\end{tabular}
\end{center}

$Ds(A)=\{c,1\}$. Let $F=\{d,1\}$. $Ds(A)/F=\{c/F,1/F\}$. For example, $b/F\in Ds(A/F)$ (since $\neg \, b/F=a/F=0/F$) and $b/F\notin Ds(A)/F$ (since $b/F\neq c/F$ and $b/F\neq 1/F$).\end{proof}

\begin{proposition}
Let $A$ be a Glivenko residuated lattice and $a\in A$. Then: $a\in {\rm Rad}(A)$ iff $\neg \, \neg \, a\in {\rm Rad}(A)$.
\label{`anegnega`}
\end{proposition}
\begin{proof}
By using in turn Proposition~\ref{`maxrad`}, (iii), Proposition~\ref{`glivinv`} and again Proposition~\ref{`maxrad`}, (iii), we get: $a\in {\rm Rad}(A)$ iff $a/Ds(A)\in {\rm Rad}(A)/Ds(A)$ iff $\neg \, \neg \, a/Ds(A)\in {\rm Rad}(A)/Ds(A)$ iff $\neg \, \neg \, a\in {\rm Rad}(A)$. Actually, according to Proposition~\ref{`maxrad`}, (iv), it was not necessary for $A$ to be Glivenko; instead, it was sufficient for $A/Ds(A)$ to satisfy the equivalence: $x\in {\rm Rad}(A/Ds(A))$ iff $\neg \, \neg \, x\in {\rm Rad}(A/Ds(A))$.\end{proof}

\begin{proposition}
Let $A$ be a Glivenko residuated lattice. Then:\linebreak ${\rm Rad}({\rm Reg}(A))={\rm Rad}(A)\cap {\rm Reg}(A)$.
\end{proposition}
\begin{proof}
Let us first prove that $\neg \, \neg \, ({\rm Rad}(A))={\rm Rad}(A)\cap {\rm Reg}(A)$. For all $a\in A$, $a\in {\rm Rad}(A)\cap {\rm Reg}(A)$ iff $a\in {\rm Rad}(A)$ and $a=\neg \, \neg \, a$, which implies $a\in \neg \, \neg \, ({\rm Rad}(A))$. Let $a\in \neg \, \neg \, ({\rm Rad}(A))$, hence $a=\neg \, \neg \, b$, with $b\in {\rm Rad}(A)$, hence, by Proposition~\ref{`anegnega`}, $\neg \, \neg \, b\in {\rm Rad}(A)$, so $a\in {\rm Rad}(A)$. But $\neg \, \neg \, ({\rm Rad}(A))\subseteq \neg \, \neg \, A={\rm Reg}(A)$, by Proposition~\ref{`glivnegneg`}. So $a\in {\rm Rad}(A)\cap {\rm Reg}(A)$. Hence the desired set equality.

Propositions~\ref{`dreg`} and~\ref{`maxrad`}, (iv), and the observation in the previous paragraph show that: ${\rm Rad}({\rm Reg}(A))={\rm Rad}(\theta (A/Ds(A)))=\theta ({\rm Rad}(A/Ds(A)))=\theta ({\rm Rad}(A)/Ds(A))=\theta (p_{A}({\rm Rad}(A)))=\neg \, \neg \, ({\rm Rad}(A))={\rm Rad}(A)\cap {\rm Reg}(A)$.\end{proof}

\begin{proposition}
Let $A$ be a Glivenko residuated lattice. Then: $A/Ds(A)$ is an MV-algebra
iff $A$ satisfies the equation: $(\neg \, a\rightarrow \neg \, b)\rightarrow \neg \, b=(\neg \, b\rightarrow \neg \, a)\rightarrow \neg \, a$.
\label{ecmv}
\end{proposition}

\begin{proof}
By Propositions \ref{`dreg`} and \ref{regmv}.\end{proof}

\begin{corollary}
For any BL-algebra $A$, $A/Ds(A)$ is an MV-algebra.
\label{abladamv}
\end{corollary}
\begin{proof}
Any BL-algebra is Glivenko (see \cite[page 164]{cito2}) and satisfies the equation from Proposition \ref{ecmv} (see \cite[page 120]{cito1}).\end{proof}

In the following we will give an example of an MTL-algebra $A$ such that $A/Ds(A)$ is not an MV-algebra, actually not even a BL-algebra.

An MTL-algebra $A$ that is not a BL-algebra and in which $Ds(A)=\{1\}$ will provide us with the desired example.

\begin{example}
We are using an example in \cite[Section 14.1.2]{ior2}.

Let $A=\{0,a,b,c,d,1\}$, with $0<a<b<c<d<1$ and the following operations:

\begin{center}
\begin{tabular}{cc}
\begin{tabular}{c|cccccc}
$\rightarrow $ & $0$ & $a$ & $b$ & $c$ & $d$ & $1$ \\ \hline
$0$ & $1$ & $1$ & $1$ & $1$ & $1$ & $1$ \\
$a$ & $d$ & $1$ & $1$ & $1$ & $1$ & $1$ \\
$b$ & $c$ & $c$ & $1$ & $1$ & $1$ & $1$ \\
$c$ & $b$ & $b$ & $c$ & $1$ & $1$ & $1$ \\
$d$ & $a$ & $a$ & $b$ & $c$ & $1$ & $1$ \\
$1$ & $0$ & $a$ & $b$ & $c$ & $d$ & $1$
\end{tabular}
&
\begin{tabular}{c|cccccc}
$\odot $ & $0$ & $a$ & $b$ & $c$ & $d$ & $1$ \\ \hline
$0$ & $0$ & $0$ & $0$ & $0$ & $0$ & $0$ \\
$a$ & $0$ & $0$ & $0$ & $0$ & $0$ & $a$ \\
$b$ & $0$ & $0$ & $0$ & $0$ & $b$ & $b$ \\
$c$ & $0$ & $0$ & $0$ & $b$ & $c$ & $c$ \\
$d$ & $0$ & $0$ & $b$ & $c$ & $d$ & $d$ \\
$1$ & $0$ & $a$ & $b$ & $c$ & $d$ & $1$
\end{tabular}
\end{tabular}
\end{center}

With these operations, $A$ becomes an MTL-algebra (even an IMTL-algebra) that is not a BL-algebra. Indeed, $b\wedge a=a\neq 0=b\odot (b\rightarrow a)$, so $A$ does not satisfy the divisibility equation. As one can see, $Ds(A)=\{1\}$, so $A/Ds(A)\cong A$, which is not a BL-algebra.
\end{example}

\section{Lifting Boolean Center}
\label{liftprop}

\hspace*{11pt} In this section we collect several results about residuated lattices with lifting Boolean center.

In what follows, we recall the definition of the residuated lattices with lifting Boolean center, that we gave in \cite{eu3}.

In the following, unless mentioned otherwise, let $A$ be a residuated lattice and let us consider the commutative diagram below, where, for all $a\in A$, $p_{A}(a)=a/Ds(A)$, $r_{A}(a)=a/{\rm Rad}(A)$ (the canonical surjections) and $\phi _{A}(a/Ds(A))=a/{\rm Rad}(A)$. Since $Ds(A)\subseteq {\rm Rad}(A)$, we have that $\phi _{A}$ is well defined and it is a morphism of residuated lattices.

\begin{center}
\begin{picture}(70,70)(0,0)
\put(2,43){$A$}
\put(12,48){\vector(1,0){38}}
\put(27,51){$p_{A}$}
\put(53,43){$A/Ds(A)$}
\put(56,40){\vector(0,-1){26}}
\put(59,25){$\phi _{A}$}
\put(10,40){\vector(4,-3){40}}
\put(14,20){$r_{A}$}
\put(53,2){$A/{\rm Rad}(A)$}
\end{picture}
\end{center}

As we have seen in Section \ref{prelim1}, it follows that we have the commutative diagram below in the category of Boolean algebras.

\begin{center}
\begin{picture}(70,70)(0,0)
\put(2,43){$B(A)$}
\put(32,48){\vector(1,0){38}}
\put(34,51){$B(p_{A})$}
\put(73,43){$B(A/Ds(A))$}
\put(76,40){\vector(0,-1){26}}
\put(79,25){$B(\phi _{A})$}
\put(30,40){\vector(4,-3){40}}
\put(14,20){$B(r_{A})$}
\put(73,2){$B(A/{\rm Rad}(A))$}
\end{picture}
\end{center}

\begin{lemma}{\rm \cite{eu3}} $B(p_{A})$ and $B(r_{A})$ are injective.
\end{lemma}

\begin{definition}
$A$ has {\em lifting Boolean center} iff $B(r_{A})$ is surjective (and hence a Boolean isomorphism).
\end{definition}

\begin{lemma}{\rm \cite{eu3}} If $A$ has lifting Boolean center then $B(\phi _{A})$ is surjective.
\label{lpphi}
\end{lemma}

\begin{proposition}
If $A$ is Glivenko, then: $A$ has lifting Boolean center iff $B(\phi _{A})$ is surjective.
\label{lpiffsurj}
\end{proposition}

\begin{proof}
The direct implication results from Lemma \ref{lpphi}. 

For the converse implication let us notice that, if $A$ is Glivenko, then in the following commutative diagram $\theta $ is an isomorphism of residuated lattices, as we have seen in Proposition \ref{`dreg`}:

\begin{center}
\begin{picture}(70,70)(0,0)
\put(2,43){$A$}
\put(12,48){\vector(1,0){38}}
\put(27,51){$p_{A}$}
\put(53,43){$A/Ds(A)$}
\put(56,40){\vector(0,-1){26}}
\put(59,25){$\theta $}
\put(10,40){\vector(4,-3){40}}
\put(10,20){$\neg \, \neg \, $}
\put(53,2){${\rm Reg}(A)$}
\end{picture}
\end{center}

In the category of Boolean algebras we have the following commutative diagram:

\begin{center}
\begin{picture}(170,70)(0,0)
\put(20,53){$B(A)$}
\put(40,50){\vector(0,-1){20}}
\put(0,38){$B(\neg \, \neg \, )$}
\put(145,38){$B(\phi _{A})$}
\put(15,18){$B({\rm Reg}(A))$}
\put(46,58){\vector(1,0){90}}
\put(75,61){$B(r_{A})$}
\put(139,53){$B(A/{\rm Rad}(A))$}
\put(139,18){$B(A/Ds(A))$}
\put(136,22){\vector(-1,0){69}}
\put(142,30){\vector(0,1){20}}
\put(90,9){$B(\theta )$}
\put(47,53){\vector(3,-1){90}}
\put(98,40){$B(p_{A})$}
\end{picture}
\end{center}

Proposition \ref{car-B(A)} says that, for all $a\in B(A)$, $a=\neg \, \neg \, a$. Let $a,b\in B(A)$ such that $\neg \, \neg \, a=\neg \, \neg \, b$. This is equivalent to $a=b$. So $B(\neg \, \neg \, )$ is injective. Now let $a\in B({\rm Reg}(A))$, so $a\in B(A)$, hence $a=\neg \, \neg \, a=B(\neg \, \neg \, )(a)$, therefore $B(\neg \, \neg \, )$ is surjective. So $B(\neg \, \neg \, )$ is a Boolean isomorphism. $B(\theta )$ is also an isomorphism, since $\theta $ is an isomorphism. Hence $B(p_{A})$ is an isomorphism, so if $B(\phi _{A})$ is surjective then $B(r_{A})$ is surjective, so that $A$ has lifting Boolean center.\end{proof}

\begin{proposition}{\rm \cite{eu3}} $A/Ds(A)$ has lifting Boolean center iff $B(\phi _{A})$ is a Boolean isomorphism.
\label{lpiffisom}
\end{proposition}

\begin{proposition}
Let $A$ be a Glivenko residuated lattice. Then: if $A/Ds(A)$ has lifting Boolean center then $A$ has lifting Boolean center.
\label{glivlp}
\end{proposition}

\begin{proof}
By Propositions \ref{lpiffisom} and \ref{lpiffsurj}.\end{proof}

\begin{proposition}{\rm \cite[Proposition 5]{figele}} Any MV-algebra has lifting Boolean center.

\label{mvlp}
\end{proposition}

\begin{corollary}
Any BL-algebra has lifting Boolean center.
\label{bllp}
\end{corollary}

\begin{proof}
Let $A$ be a BL-algebra. Then: $A$ is Glivenko (\cite{cito2}) and, by Corollary \ref{abladamv} and Propositions \ref{mvlp} and \ref{glivlp}, $A/Ds(A)$ is an MV-algebra, so $A/Ds(A)$ has lifting Boolean center, hence $A$ has lifting Boolean center.\end{proof}

\begin{corollary}
Any Glivenko residuated lattice $A$ that satisfies the fo\-l\-lo\-wing equation: for all $a,b\in A$, $(\neg \, a\rightarrow \neg \, b)\rightarrow \neg \, b=(\neg \, b\rightarrow \neg \, a)\rightarrow \neg \, a$, has lifting Boolean center.
\end{corollary}
\begin{proof}
By Propositions \ref{ecmv}, \ref{mvlp} and \ref{glivlp}.\end{proof}

See \cite{eu3} for an example of a residuated lattice that does not have lifting Boolean center. 

\section{Simple and Quasi-local Residuated Lattices}
\label{locsemiloc}

\hspace*{11pt} In this section we study the classes of local, semilocal, simple and quasi-local residuated lattices, as well as the relations between these classes. Semilocal and quasi-local residuated lattices are more general concepts than local residuated lattices. Quasi-local residuated lattices extend quasi-local MV-algebras and quasi-local BL-algebras (structures that characterize weak Boolean products of local MV-algebras, respectively weak Boolean products of local BL-algebras) (by \cite{dngele-2001}). We conjecture that this result remains valid for a much larger class of residuated lattices; the cha\-rac\-te\-ri\-za\-tion of this class is still an open problem.

\begin{definition}
A residuated lattice is said to be {\em local} iff it has exactly one maximal filter.
\end{definition}

\begin{proposition}
If $A$ is a local residuated lattice, then the unique maximal filter of $A$ is $D(A)=\{a\in A|ord(a)=\infty \}$, so ${\rm Rad}(A)=D(A)$.
\end{proposition}

\begin{proof}
This is part of \cite[Theorem 4.3]{ciu}.\end{proof}

It is known that simple residuated lattices are characterized by the following result.

\begin{proposition}
A residuated lattice $A$ is simple iff, for every $a\in A$, \[a\ne 1\Ra ord(a)<\infty.\]
\end{proposition}

\begin{proposition}
Any simple residuated lattice is local.
\end{proposition}

\begin{proof}
It is immediate that the only filters of a simple residuated lattice $A$ are $\{1\}$ and $A$, so the only maximal filter is $\{1\}$.\end{proof}

\bprop\label{locally-finite}{\rm \cite[Theorem 1.60, page 35]{pic}} Let $A$ be a residuated lattice and $M$ a proper filter of $A$. Then the following are equivalent:\\(i) $A/M$ is a simple residuated lattice;\\(ii) $M$ is a maximal filter. 
\eprop 

\begin{definition}
A residuated lattice is said to be {\em semilocal} iff it has only a finite number of maximal filters.
\end{definition}

It is obvious that any finite residuated lattice is semilocal and any local residuated lattice is semilocal.


\bprop\label{MaxA-Rad}
For any residuated lattice $A$,
\[|{\rm Max}(A)|=|{\rm Max}(A/{\rm Rad}(A))|.\]
Hence, $A$ is semilocal iff $A/{\rm Rad}(A)$ is semilocal, and $A$ is local iff $A/{\rm Rad}(A)$ is local.
\eprop

\begin{proof}
In a manner similar to the proof of Proposition \ref{`maxrad`}, (v), one can show that there is a bijection between ${\rm Max}(A)$ and ${\rm Max}(A/{\rm Rad}(A))$.\end{proof}

Following the analogous definitions for BL-algebras from \cite{dngele-2001}, we define the quasi-local residuated lattices and the primary and quasi-primary filters of a residuated lattice.

\begin{definition}
A proper filter $F$ of a residuated lattice $A$ is called {\em primary} iff, for all $a,b\in A$, $\neg \, (a\odot b)\in F$ implies that there exists $n\in {\rm I\! N}^{*}$ such that $\neg \, a^{n}\in F$ or $\neg \, b^{n}\in F$.
\end{definition}

\begin{definition}
A residuated lattice $A$ is called {\em quasi-local} iff, for any $a\in A$, there exist $e\in B(A)$ and $n\in {\rm I\! N}^{*}$ such that $a^{n}\odot e=0$ and $(\neg \, a)^{n}\odot (\neg \, e)=0$.
\label{qloc}
\end{definition}

\begin{definition}
A proper filter $F$ of a residuated lattice $A$ is called {\em quasi-primary} iff, for all $a,b\in A$, $\neg \, (a\odot b)\in F$ implies that there exist $u\in A$ and $n\in {\rm I\! N}^{*}$ such that $u\vee \neg \, u\in B(A)$, $\neg \, (a^{n}\odot u)\in F$ and $\neg \, (b^{n}\odot \neg \, u)\in F$.
\end{definition}

The proposition below, in its variant for BL-algebras, is \cite[Proposition 4.10]{dngele-2001}; its proof is also valid for residuated lattices.

\begin{proposition}
Any primary filter of a residuated lattice is quasi-primary.
\end{proposition}

\begin{proposition}{\rm \cite[Corollary 4.4]{ciu}} If $A$ is a local residuated lattice, then: for any $a\in A$, $ord(a)<\infty $ or $ord(\neg \, a)<\infty $.
\label{locord}
\end{proposition}

\begin{proposition}
Any local residuated lattice is quasi-local.
\end{proposition}

\begin{proof}
Same as the proof of the first implication from \cite[Proposition 4.9]{dngele-2001} (see Proposition \ref{locord}).\end{proof}

\begin{proposition}
If $A$ is a local residuated lattice, then $B(A)=\{0,1\}$.

\end{proposition}

\begin{proof}

Same as the proof of \cite[Proposition 4.4]{dngele-2001} (see Proposition \ref{locord}).\end{proof}

\begin{proposition}
Let $A$ be a quasi-local residuated lattice and $F$ a filter of $A$. Then $A/F$ is quasi-local.
\label{A/Fqloc}
\end{proposition}

\begin{proof}
Obviously, it is sufficient to prove that $B(A/F)\supseteq B(A)/F$. By Proposition \ref{B(A)=}, $B(A/F)=\{e/F|e\in A,\ e\vee \neg \, e\in F\}\supseteq \{e/F|e\in A,\ e\vee \neg \, e=1\}=\{e/F|e\in B(A)\}=B(A)/F$.\end{proof}

\begin{proposition}
Let $A$ be a residuated lattice and let us consider the following properties:

\noindent (i) $A$ is quasi-local;

\noindent (ii) $A/Ds(A)$ is quasi-local.

Then:

\noindent (I) (i) implies (ii);

\noindent (II) (ii) does not imply (i);

\noindent (III) if $A$ is Glivenko and it satisfies the equation: for all $a,b\in A$, $(\neg \, a\rightarrow \neg \, b)\rightarrow \neg \, b=(\neg \, b\rightarrow \neg \, a)\rightarrow \neg \, a$, then (ii) implies (i).

\end{proposition}

\begin{proof}
\noindent (I) By Proposition \ref{A/Fqloc}.

\noindent (II) The following example from \cite[Section 16]{ior2} proves this property: $A=\{0,a,b,c,d,e,1\}$, with the residuated lattice structure shown below:

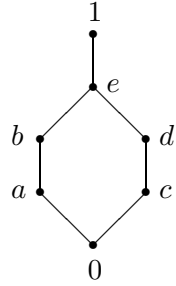
\begin{center}
\begin{picture}(80,115)(0,0)
\put(40,15){\circle*{3}}
\put(38,2){$0$}
\put(20,35){\circle*{3}}
\put(9,33){$a$}
\put(20,55){\circle*{3}}
\put(9,53){$b$}
\put(60,35){\circle*{3}}
\put(65,33){$c$}
\put(60,55){\circle*{3}}
\put(65,53){$d$}
\put(40,75){\circle*{3}}
\put(45,73){$e$}
\put(40,95){\circle*{3}}
\put(38,100){$1$}
\put(40,15){\line(1,1){20}}
\put(40,15){\line(-1,1){20}}
\put(20,35){\line(0,1){20}}
\put(60,35){\line(0,1){20}}
\put(40,75){\line(-1,-1){20}}
\put(40,75){\line(1,-1){20}}
\put(40,75){\line(0,1){20}}
\end{picture}
\end{center}

\begin{center}
\begin{tabular}{cc}
\begin{tabular}{c|ccccccc}
$\rightarrow $ & $0$ & $a$ & $b$ & $c$ & $d$ & $e$ & $1$ \\ \hline
$0$ & $1$ & $1$ & $1$ & $1$ & $1$ & $1$ & $1$ \\
$a$ & $d$ & $1$ & $1$ & $d$ & $d$ & $1$ & $1$ \\
$b$ & $d$ & $e$ & $1$ & $d$ & $d$ & $1$ & $1$ \\
$c$ & $b$ & $b$ & $b$ & $1$ & $1$ & $1$ & $1$ \\
$d$ & $b$ & $b$ & $b$ & $e$ & $1$ & $1$ & $1$ \\
$e$ & $0$ & $b$ & $b$ & $d$ & $d$ & $1$ & $1$ \\
$1$ & $0$ & $a$ & $b$ & $c$ & $d$ & $e$ & $1$
\end{tabular}
& \hspace*{11pt}
\begin{tabular}{c|ccccccc}
$\odot $ & $0$ & $a$ & $b$ & $c$ & $d$ & $e$ & $1$ \\ \hline
$0$ & $0$ & $0$ & $0$ & $0$ & $0$ & $0$ & $0$ \\
$a$ & $0$ & $a$ & $a$ & $0$ & $0$ & $a$ & $a$ \\
$b$ & $0$ & $a$ & $a$ & $0$ & $0$ & $a$ & $b$ \\
$c$ & $0$ & $0$ & $0$ & $c$ & $c$ & $c$ & $c$ \\
$d$ & $0$ & $0$ & $0$ & $c$ & $c$ & $c$ & $d$ \\
$e$ & $0$ & $a$ & $a$ & $c$ & $c$ & $e$ & $e$ \\
$1$ & $0$ & $a$ & $b$ & $c$ & $d$ & $e$ & $1$
\end{tabular}
\end{tabular}
\end{center}

One can see that $B(A)=\{0,1\}$. It is easy to check that $A$ is not quasi-local. Just take $a$ instead of $a$ in Definition \ref{qloc}.

$Ds(A)=\{e,1\}$. As the table of the operation $\leftrightarrow $ shows, $A/Ds(A)=\{0/Ds(A),a/Ds(A)=b/Ds(A),c/Ds(A)=d/Ds(A),e/Ds(A)=1/Ds$\linebreak $(A)\}$. U\-sing Lemma \ref{`leqres`}, (i), we get the following lattice structure for $A/Ds(A)$:

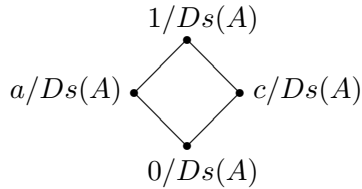
\begin{center}
\begin{picture}(180,75)(0,0)
\put(90,15){\circle*{3}}
\put(75,2){$0/Ds(A)$}
\put(70,35){\circle*{3}}
\put(23,33){$a/Ds(A)$}
\put(110,35){\circle*{3}}
\put(115,33){$c/Ds(A)$}
\put(90,55){\circle*{3}}
\put(75,60){$1/Ds(A)$} 
\put(90,15){\line(1,1){20}}
\put(90,15){\line(-1,1){20}}
\put(90,55){\line(-1,-1){20}}
\put(90,55){\line(1,-1){20}}
\end{picture}
\end{center}

Hence $B(A/Ds(A))=A/Ds(A)$. It is easy to verify that $A/Ds(A)$ is quasi-local.

\noindent (III) Assume that $A/Ds(A)$ is quasi-local and that $A$ is Glivenko and it satisfies the equation: for all $a,b\in A$, $(\neg \, a\rightarrow \neg \, b)\rightarrow \neg \, b=(\neg \, b\rightarrow \neg \, a)\rightarrow \neg \, a$. This last condition is equivalent to the fact that ${\rm Reg}(A)$ is an MV-algebra (see Proposition \ref{regmv}). The fact that $A$ is Glivenko implies that $A/Ds(A)$ and ${\rm Reg}(A)$ are isomorphic residuated lattices (see Proposition \ref{`dreg`}), hence ${\rm Reg}(A)$ is quasi-local. Let $a\in A$. By Lemma \ref{`leqres`}, (iii), $\neg \, a\in {\rm Reg}(A)$, so there exist $e\in B({\rm Reg}(A))$ and $n\in {\rm I\! N}^{*}$ such that $(\neg \, a)^{n}\odot e=0$ and $(\neg \, \neg \, a)^{n}\odot (\neg \, e)=0$. This last equality and Lemma \ref{`leqres`}, (ii) and (v), ensure us that $a^{n}\odot (\neg \, e)=0$. The fact that $(\neg \, a)^{n}\odot e=0$ and Proposition \ref{car-B(A)} show that $(\neg \, a)^{n}\odot (\neg \, \neg \, e)=0$. Since $e\in B(A)$, we have that $\neg \, e=e^{\prime }\in B(A)$. So $A$ is quasi-local.\end{proof}

%
%


\section{Acknowledgements}

\hspace*{11pt} I thank Professor George Georgescu for offering me this research theme and for numerous suggestions concerning the tackling of these mathematical problems.

\noindent Claudia Mure\c{s}an\\University of Bucharest, Faculty of Mathematics and Computer Science, Academiei 14, RO 010014, Bucharest, Romania\\E-mail: c.muresan@yahoo.com
\end{document}